\documentclass[11pt]{amsart}
\usepackage[all]{xy}
\usepackage{amssymb}
\usepackage[mathscr]{euscript}
\usepackage{pb-diagram}
\newtheorem{thm}{Theorem}
\newtheorem{prop}[thm]{Proposition \!\!}
\newtheorem{cor}[thm]{Corollary \!\!}
\newtheorem{lem}[thm]{Lemma \!\!}

\newcommand{\bbz}{\mathbb{Z}}
\newcommand{\bbn}{\mathbb{N}}
\newcommand{\bbq}{\mathbb{Q}}
\newcommand{\bbp}{\mathbb{P}}
\newcommand{\bbc}{\mathbb{C}}

\newcommand{\bbf}{\mathbb{F}}

\newcommand{\age}{\mathrm{age}}

\newcommand{\Lie}{\mathrm{Lie}}
\newcommand{\Aut}{\mathrm{Aut}}

\newcommand{\Jac}{\mathrm{Jac}}
\newcommand{\PSL}{\mathrm{PSL}}

\newcommand{\mat}{\begin{pmatrix}}
\newcommand{\emat}{\end{pmatrix}}

\author[Bo-Hae Im]{Bo-Hae Im}
\address{Department of Mathematics, Chung-Ang University, 221, Heukseok-dong, Dongjak-gu, Seoul, 156-756, South Korea}
\email{bohaeim@gmail.com, imbh@cau.ac.kr}

\author[Michael Larsen]{Michael Larsen}
\address{Department of Mathematics, Indiana University, Bloomington,
Indiana 47405, USA}
\email{mjlarsen@indiana.edu}

\date{\today}
\thanks{Bo-Hae Im was partially supported by the National Research Foundation of Korea Grant funded by the Korean Government(MEST) (NRF-2011-0015557). Michael Larsen was partially supported by NSF grant DMS-1101424.}

\begin{document}

\thispagestyle{empty}
\title[Rational curves on $A/G$]
{Rational curves on quotients of abelian varieties by  finite groups}

\begin{abstract} In \cite{KL}, it is proved that the quotient of an abelian variety $A$ by a finite order automorphism $g$ is uniruled
if and only if some power of $g$ satisfies a numerical condition $0<\age(g^k)<1$.  In this paper, we show that $\age(g^k)=1$ is enough to
guarantee that $A/\langle g\rangle$ has at least one rational curve.
\end{abstract}

 \maketitle

\section{Introduction}

Let $G$ be a finite group of automorphisms of an abelian variety $A/\bbc$.  It is a classical result \cite[II. \S1]{La} that $A$ itself cannot contain a rational curve.
For $|G| > 1$, there may or may not be rational curves on $A/G$. For general abelian varieties, $\Aut(A)= \pm1$, and Pirola proved \cite{P} that for $A$ sufficiently general and of dimension at least three, $A/\pm1$ has no rational curves. At the other extreme, regarding $A=E^n$ as the set of $n+1$-tuples of points on the elliptic curve $E$ which sum to $0$, $A/S_{n+1}$ can be interpreted as the set of effective divisors linearly equivalent to $(n+1)[0]$ and, as such, is just $\bbp^n$. More generally, Looijenga has shown  \cite{L} that the quotient of $E^n$ by the Weyl group of a root system of rank $n$ is a weighted projective space.

Rational curves on $A/G$ over a field $K$ are potentially a source of rational points over $G$-extensions of $K$.  For instance, the method \cite{Larsen} for finding pairs $a,b\in \bbq^\times$ such that the quadratic twists
$E_a$, $E_b$, and $E_{ab}$ all have positive rank amounts to finding a rational curve on $E^3/(\bbz/2\bbz)^2$.   Likewise, the theorem of Looijenga cited above gives for each elliptic curve $E$ over a number field $K$ and for each Weyl group $W$, a source of $W$-extensions $L_i$ of $K$ such that the representation of $W$ on each $E(L_i)\otimes\bbq$ contains the reflection representation.  On the other hand, the result of Pirola cited above dims the hope of using geometric methods to show that every abelian variety over a number field $K$ gains rank over infinitely many quadratic extensions of $K$.  Thus, it is desirable from the viewpoint of arithmetic to understand when $A/G$ can be expected to have a rational curve over a given field $K$, and to begin with, one would like to know when $A/G$ has a rational curve over $\bbc$.

Any automorphism $g$ of an abelian variety $A$ defines an invertible linear transformation (also denoted $g$) on $\Lie(A)$.
If $g$ is of finite order, there exists a unique sequence of rationals $0\le x_1\le x_2\le \cdots\le x_n < 1$ such that
the eigenvalues of $g$ are $e(x_1),\ldots,e(x_n)$, where $e(x) := e^{2\pi i x}$.  We say $g$ is of \emph{type} $(x_1,\ldots,x_n)$. Following
Koll\'ar and Larsen \cite{KL}, we write $\age(g) = x_1+\cdots+x_n$.
For instance, $\age(g) = 1/2$ for every reflection $g$.  The main result of \cite{KL} asserts that
$A/G$ is uniruled if and only if $0<\age(g)<1$ for some $g\in G$.  In this paper, we prove that to find a single rational curve in $A/G$, it sufices that $\age(g) \le 1$.

Since we need only consider the case $\age(g) = 1$, we first classify all types of weight $1$.  This requires a combinatorial analysis, which we carried out using a computer algebra system to minimize the risk of an oversight.
There are thirty-five cases (see Table 2 below), and our strategy for finding rational curves depends on case analysis.  Abelian surfaces play a special role, since here we can use known results on K3 surfaces.
The other key  idea is to find a non-singular projective curve $X$ on which $G$ acts with quotient
$\bbp^1$ and a $G$-equivariant map from $X$ to $A$, or, equivalently, a $G$-homomorphism from the Jacobian variety of $X$ to $A$.

We would like to thank Yuri Tschinkel and Alessio Corti for helpful comments on earlier versions of this paper.

\section{Classifying types}

If $A = V/\Lambda$, then the Hodge decomposition $\Lambda\otimes\bbc \cong V\oplus \bar V$ respects the action of $\Aut(A)$.
Therefore, if $g$ is of finite order with eigenvalues $e(x_1),\ldots,e(x_n)$, then the multiset
\begin{equation}
\tag{$\ast$}\{e(x_1),\ldots,e(x_n),e(-x_1),\ldots,e(-x_n)\}
\end{equation}
is $\Aut(\bbc)$-stable.
By a \emph{type}, we mean a multiset $\{x_1,\ldots,x_n\}$ with $x_i\in [0,1)$ such that the multiset ($\ast$) is $\Aut(\bbc)$-stable.
Equivalently, a type can be identified with a finitely supported function $f\colon \bbq/\bbz\to \bbn$ such that $f(x)+f(-x)$ depends only on the order of $x$ in $\bbq/\bbz$.
By the \emph{weight} of $\{x_1,\ldots,x_n\}$, we mean the sum $x_1+\cdots+x_n$, so that $\age(g)$ is the weight of the type of $g$.

A type is \emph{reduced} if $0$ does not appear, and the reduced type of a given type is obtained by discarding all copies of $0$.
The \emph{sum} of types is the union in the sense of multisets; at the level of associated functions on $\bbq/\bbz$ it is the usual sum.
A type which is not the sum of non-zero types is \emph{primitive}.
All the elements of a primitive type appear with multiplicity one, and they all have the same denominator.
Every type can be realized (not necessarily uniquely) as a sum of primitive types; if the weight of the type is $1$, each of the primitive types has weight $\le 1$,
so our first task is to classify primitive types with weight $\le 1$.

A primitive type $X$ of denominator $n\ge 2$ consists of fractions $a_i/n$ where $0<a_i<n$, and $(a_i,n)=1$.
Moveover, if $n\ge 3$, for each positive integer $a<n$ prime to $n$, exactly one of $a/n$ and $1-a/n$ belongs to $X$.  If $a<b<n/2$ and $1-a/n\in X$,
then the weight of $X$ exceeds $1$ since either $b/n$ or $1-b/n$ belongs to $X$.
Thus, if $0<a<n/2$ and $1-a/n\in X$, then $a$ must be the largest integer in $(0,n/2)$ prime to $n$.

\begin{lem}\label{primineq} If $\phi(n) > 24$, then  $$\sum_{x\in S_n} \min(x,n-x) > 2n,$$ where $S_n$ is the set of  positive integers $< n$ and prime to $n$. Moreover, the largest integer $n$ such that $\phi(n)\le 24$ is $90$.
\end{lem}

\begin{proof} Note that
$$\min(x,n-x) > \frac{x(n-x)}{n} = \frac{n^2-x^2-(n-x)^2}{2n}.$$
In order to prove the first statement, we want to prove if $\phi(n)>24$, then
$$\sum_{x\in S_n} \Big(n^2-x^2-(n-x)^2\Big) > 4n^2,$$ or equivalently,
$$\phi(n)n^2  - 2\sum_{x\in S_n} x^2-4n^2>0.$$
By  M\"{o}bius inversion, one can prove that
$$\sum_{x\in S_n} x^2 = \frac{\phi(n) n^2}{3} +(-1)^{d_n}\frac{\phi(f(n)) n}6,$$
where $f(n)$ denotes the larges squarefree divisor of $n$ and $d_n $ is the number of distinct prime divisors of $n$.  Thus, if $\phi(n) > 24$, then $\phi(n)>24\geq 12 \frac{n}{n-1}$, so $(n-1)\phi(n)-12n >0$ and since $\phi(f(n))\leq \phi(n)$,
\begin{align*}
\phi(n)n^2 - 2\sum_{x\in S_n} x^2-4n^2 &\geq \Big(\frac{\phi(n)}{3}-4\Big)n^2-\frac{\phi(n)n}{3} \\
&=\frac{n((n-1)\phi(n)-12n)}{3}>0,
\end{align*}
which is the desired inequality.

For the second statement, if $\phi(n)\le 24$ and $p$ is a prime factor of $n$, then $\phi(p)=p-1\leq\phi(n)\leq 24$. Hence $p\le 23$.  Writing
$$n=2^{n_2}3^{n_3}5^{n_5}7^{n_7}11^{n_{11}}13^{n_{13}}17^{n_{17}}19^{n_{19}}23^{n_{23}},$$
we have 
$$0\leq n_2\leq 5,0\leq n_3\leq 3,0\leq n_5\leq 2,$$
and $0\leq n_i\leq 1$ for $7\le i\le 23$.
Case analysis now shows $n\leq 90$.

\end{proof}

\begin{prop}\label{primlist} There are $28$  primitive types with weight $ \leq 1$:
\end{prop}
\begin{center}
\begin{tabular}{|l|c|c|c|}
\hline
\#& n & \text{primitive types} &  \text{weight}\\\hline
1&  2 & 1/2 & 1/2\\\hline
2&  3 & 1/3 & 1/3\\
3&    & 2/3 & 2/3\\\hline
4&  4 & 1/4 & 1/4 \\
5&    & 3/4 & 3/4\\\hline
6&  5 & 1/5, 2/5 & 3/5 \\
7&    & 1/5, 3/5 & 4/5 \\\hline
8&  6 & 1/6 &1/6\\
9&    & 5/6 & 5/6\\\hline
10&  7 & 1/7, 2/7, 3/7 & 6/7\\
11&    & 1/7, 2/7, 4/7 & 1\\\hline
12&  8 & 1/8, 3/8 & 1/2 \\
13&    & 1/8, 5/8 & 3/4 \\\hline
14&  9 & 1/9, 2/9, 4/9 & 7/9 \\
15&    & 1/9, 2/9, 5/9 & 8/9\\\hline
16&  10 & 1/10, 3/10 & 2/5 \\
17&    & 1/10, 7/10 & 4/5 \\\hline
18&  12 & 1/12, 5/12 & 1/2\\
19&    & 1/12, 7/12 & 2/3\\\hline
20&  14 & 1/14, 3/14, 5/14 & 9/14\\
21&    & 1/14, 3/14, 9/14 & 13/14\\\hline
22&  15 & 1/15, 2/15, 4/15, 7/15 & 14/15\\
23&    &1/15,2/15, 4/15, 8/15 & 1\\\hline
24&  16 & 1/16, 3/16, 5/16, 7/16 & 1\\\hline
25&  18 & 1/18, 5/18, 7/18 & 13/18\\
26&    & 1/18, 5/18, 11/18 & 17/18 \\\hline
27&  20 & 1/20, 3/20, 7/20, 9/20 & 1\\\hline
28&  24 & 1/24, 5/24, 7/24, 11/24 & 1\\
  \hline
\end{tabular}
\vskip 4pt
Table 1
\end{center}
\begin{proof} For $n\ge 3$, the weight of a primitive type of denominator $n$ is at least
$$\sum_{\{x\in S_n\mid x<n/2\}}\frac xn \ge \frac 1{2n}\sum_{x\in S_n} \min(x,n-x).$$
By Lemma~\ref{primineq}, it suffices to carry out an exhaustive search  up to $n = 90$.
\end{proof}
\begin{lem} There are $35$ types with $\age ~~1$ of automorphisms given in Table 2 below.
\end{lem}
\begin{center}
\begin{tabular}{|c|c|c|c|}
\hline
\#&  n & \text{types} &  \text{notes}\\\hline
1&  2 & 1/2, 1/2 & Prop.~\ref{kummer} \\\hline
2&  3 & 1/3, 1/3, 1/3 &  Th.~\ref{small-factor}\\
3&   & 1/3, 2/3 &  Prop.~\ref{cube}\\\hline
4&  4 & 1/4, 1/4, 1/4, 1/4 &    Th.~\ref{small-factor}\\
5&   & 1/4, 1/4, 2/4 & $g^2\to\#1$ \\
6&   & 1/4, 3/4 & $g^2\to\#1$  \\\hline
7&  6 & 1/6, 1/6, 1/6, 1/6, 1/6, 1/6 &    Th.~\ref{small-factor}\\
8&  & 1/6, 1/6, 1/6, 1/6, 2/6  & Th.~\ref{small-factor} \\
9&   & 1/6, 1/6, 1/6, 3/6  &  $g^2\to\#2$ \\
10&  & 1/6, 1/6,  4/6  &  $g^3\to\#1$ \\
11&    & 1/6, 5/6  &  $g^3\to\#1$ \\
12&   & 1/6, 2/6, 3/6  &  $g^3\to\#1$  \\
13&    & 1/6, 1/6,  2/6, 2/6  &  $g^3\to\#1$  \\\hline
14&   7 & 1/7, 2/7, 4/7 & Cor.~\ref{seven}\\\hline
15&   8 & 1/8, 2/8, 5/8  &   $g^4\to\#1$\\
16&   &  1/8, 3/8, 4/8 & $g^4\to\#1$ \\
17&   & 1/8,  1/8, 3/8, 3/8 & Th.~\ref{small-factor} \\
18&    & 1/8, 2/8, 2/8, 3/8 & $g^4\to\#1$  \\\hline
19&   10 &  1/10, 2/10, 3/10, 4/10 & $g^5\to\#1$ \\\hline
20&   12 & 4/12, 2/12, 1/12, 5/12& $g^6\to\#1$\\
21&    & 4/12, 3/12, 3/12, 2/12 &$g^6\to\#1$\\
22&    & 6/12, 1/12, 5/12 &$g^6\to\#1$\\
23&    & 3/12, 3/12, 1/12, 5/12 &$g^4\to \#3$\\
24&    & 2/12, 2/12, 2/12, 1/12, 5/12 & $g^6\to\#1$\\
25&    & 1/12, 1/12, 5/12, 5/12 & Th.~\ref{small-factor} \\
26&    & 4/12, 1/12, 7/12 & $g^6\to\#1$ \\
27&    & 2/12, 2/12, 1/12, 7/12 & $g^6\to\#1$ \\
28&    & 3/12, 3/12, 2/12, 2/12, 2/12 & $g^6\to\#1$\\\hline
29&   15 & 1/15, 2/15, 4/15, 8/15 & Cor.~\ref{fifteen}\\\hline
30&   16 & 1/16, 3/16, 5/16, 7/16 & Cor.~\ref{mult-one}\\ \hline
31&   20 & 1/20, 3/20, 7/20, 9/20 & Cor.~\ref{mult-one}\\\hline
32&   24 & 1/24, 5/24, 7/24, 11/24 & Cor.~\ref{mult-one}\\
33&    & 8/24, 4/24, 3/24, 9/24 & $g^{12}\to\#1$\\
34&    & 3/24, 9/24, 2/24, 10/24 & $g^{12}\to\#1$\\
35&    & 4/24, 4/24, 4/24, 3/24, 9/24 & $g^{12}\to\#1$ \\
  \hline
\end{tabular}
\vskip 4pt
Table 2
\end{center}

\begin{proof} Let $[a_i]$ be a formal variable representing the $i$th primitive type in Table $1$,  and let  $w_i$ denote the weight of the $i$th type.
A monomial $\prod a_i^{m_i}$ stands for a sum of primitive types in which the $i$th type appears $m_i$ times.
The g.c.d. of the denominators of the $w_i$  is $5040$. Let $y=x^{1/5040}$, so $(1-[a_i]x^{w_i})^{-1}$ is a power series in $y$ for every $i$.
By \textsc{Maple 13},   the coefficient of $y^{5040}$ in the product
$$\prod_{i=1}^{28} (1-[a_i]x^{w_i})^{-1}=\prod_{i=1}^{28} (1-[a_i]y^{w_i\cdot 5040})^{-1},$$
is
%
${{\it a1}}^{2}+{{\it a2}}^{3}+{\it a2}\,{\it a3}+{\it a4}\,{\it a5}+{
\it a8}\,{\it a9}+{\it a4}\,{\it a13}+{\it a6}\,{\it a16}+{{\it a4}}^{
2}{\it a1}+{\it a28}+{\it a27}+{\it a23}+{\it a24}+{\it a11}+{\it a19}
\,{\it a2}+{\it a19}\,{{\it a8}}^{2}+{\it a18}\,{\it a1}+{\it a18}\,{{
\it a4}}^{2}+{\it a18}\,{{\it a8}}^{3}+{\it a18}\,{\it a12}+{{\it a18}
}^{2}+{\it a12}\,{\it a1}+{\it a12}\,{{\it a4}}^{2}+{\it a12}\,{{\it
a8}}^{3}+{{\it a12}}^{2}+{{\it a8}}^{2}{{\it a2}}^{2}+{{\it a8}}^{2}{
\it a3}+{{\it a8}}^{3}{\it a1}+{{\it a8}}^{3}{{\it a4}}^{2}+{{\it a8}}
^{4}{\it a2}+{{\it a8}}^{6}+{{\it a4}}^{4}+{\it a18}\,{\it a8}\,{\it
a2}+{\it a12}\,{\it a8}\,{\it a2}+{\it a8}\,{\it a1}\,{\it a2}+{\it a8
}\,{\it a2}\,{{\it a4}}^{2}$.

Each monomial in this sum corresponds to an entry in Table 2.
\end{proof}

\section{Rational curves in $A/\langle g\rangle$}

In this section we explain how to find rational curves on $A/\langle g\rangle$ in each case in Table 2.

\begin{lem}
If $A/\langle g^n\rangle$ has a rational curve for some positive integer $n$, then $A/\langle g\rangle$ has a rational curve.
\end{lem}

\begin{proof}
The morphism $A/\langle g^n\rangle\to A/\langle g\rangle$ is finite, so the image of a rational curve is again a rational curve.
\end{proof}

\begin{prop}
\label{reduce}
Let $A$ be an abelian variety and $g$ an automorphism of finite order.  Suppose that for every abelian variety $B$ and finite-order automorphism $h\in \Aut(B)$
whose type is the reduced type of $g$, $B/\langle h\rangle$ has a rational curve.  Then $A/\langle g\rangle$ has a rational curve.
\end{prop}

\begin{proof}
Let $B$ denote the image of $1-g$ acting on $A$.  Then $B$ is an abelian subvariety of $A$, and $g$ restricts to an automorphism $h$ of $B$
whose type is the reduced type of $g$.  As $B/\langle h\rangle \subset A/\langle g\rangle$ has a rational curve, the same is true of $A/\langle g\rangle$.
\end{proof}

The following proposition is well known.

\begin{prop}\label{kummer}
If $A$ is an abelian surface, then $A/\pm 1$ has a rational curve.
\end{prop}

\begin{proof}
Resolving the $16$ singularities of $A/\pm 1$, we obtain a K3 surface with Picard number $\ge 16\ge 5$.  By work of Bogomolov and Tschinkel \cite{BT},  any such surface is either elliptic or has infinite automorphism
group and in either case has infinitely many rational curves, all but finitely many of which lie on $A/\pm 1$.
\end{proof}

Note that Proposition~\ref{kummer} covers not only case \#1 in Table 2 but twenty other cases as well, namely those (indicated in the ``notes'' column) for which the reduced type of some power of $g$ is $(1/2,1/2)$.

\begin{prop}\label{cube} Let $A\cong V/\Lambda$ be an abelian surface with an automorphism $g$ of type $(1/3,2/3)$.  Then $A/\langle g\rangle$ contains a rational curve.
\end{prop}

\begin{proof}
Let $G=\langle g\rangle$ and $X=A/G$. Regarding  $1-g$  as an isogeny of $A$, the number of fixed points of $g$ is
$$\deg(1-g)=\#\ker(1-g)=\det(1-g|\Lambda )=3^2=9.$$
These are singularities of  type $A_2$, since under $(x,y)\mapsto(\omega x, \omega^2 y)$ where $\omega$ is a cube root of unity, the invariants are generated by $X=x^3, Y=y^3, Z=xy$, and so
    $$\bbc[[x,y]]^G = \bbc[[X,Y,Z]]/(XY-Z^3).$$
This is isomorphic to $\bbc[[x, y, z]]/(x^2 + y^2 + z^3)$,
    which has a Du Val singularity of type $A_2$ (see \cite[Ch.4,~4.2]{R}).

Consider the  minimal resolution $f: Y\rightarrow X$, for which the $9$ exceptional divisors $Y_i$ each consists of two projective lines $D_{i,1}$ and $D_{i,2}$ intersecting at one point.
The canonical divisor of $Y$ is $K_Y=f^*K_X=0$.  Hence $Y$ is a K$3$-surface of Picard number $\ge 18$, and again by \cite{BT}, we deduce that $X$ has infinitely many rational curves..
\end{proof}

\begin{thm}
\label{disjoint}
Let $B$ be an abelian variety and $h\in \Aut(B)$ an automorphism of finite order such that $h$ and $h^{-1}$ have disjoint types.
Let $A$ be an abelian variety with a finite order automorphism $g$ whose type is contained in that of $h$.
If there is a rational curve on $B/\langle h\rangle$ then there is a rational curve on $A/\langle g\rangle$.
\end{thm}

\begin{proof}
It suffices to prove that there exists a surjective homomorphism $p\colon B\to A$ such that the diagram
$$\xymatrix{B\ar[r]^p\ar[d]_h&A\ar[d]^g \\ B\ar[r]^p&A}$$
commutes.  Writing $B = \Lie(B)/\Lambda_B$ and $A = \Lie(A)/\Lambda_A$, the goal is to find a surjective $\bbc[t]$-linear map $\phi\colon \Lie(B)\to \Lie(A)$
(where $t$ acts as $g$ on $\Lie(A)$ and as $h$ on $\Lie(B)$)
such that $\phi(\Lambda_B)\subset \Lambda_A$.  If $\psi$ is a surjective $\bbc[ t]$-linear map $\Lie(B)\to \Lie(A)$
such that $\psi(\Lambda_B\otimes\bbq) = \Lambda_A\otimes  \bbq$, then we can define $\phi := n\psi$ for $n$ a sufficiently divisible positive integer.  It suffices, therefore, to find $\psi$
with the desired properties.

As the type of $A$ is a subset of the type of $B$, there exists a surjective $\bbq[t]$-linear map $T\colon\Lambda_B\otimes\bbq\to \Lambda_A\otimes  \bbq$.
Extending scalars to $\bbc$, $T\otimes 1$ maps $\Lambda_B\otimes\bbc = \Lie(B) \oplus \overline{\Lie(B)}$ to $\Lambda_A\otimes \bbc = \Lie(A) \oplus \overline{\Lie(A)}$.
The type of $g$ acting on $\overline{\Lie(A)}$ is the same as the type of $g^{-1}$ acting on  $\Lie(A)$ and therefore disjoint from the type of $g$ acting on $\Lie(A)$, and the same is true for the type of $h$ acting on $\Lie(B)$.
As $\Lie(B)$ and $\Lie(A)$ are direct sums of certain $t$-eigenspaces of $\Lambda_B\otimes\bbc$ and $\Lambda_A\otimes\bbc$ respectively and
as the spectrum of $t$ acting on $\Lie(A)$ is the intersection of the spectra of $t$ acting on $\Lie(B)$ and on $\Lambda_A\otimes\bbc$, it follows that
$T\otimes 1$ maps $\Lie(B)$ to $\Lie(A)$.  The restriction of $\psi$ to $\Lie(B)$ is therefore the desired map.

\end{proof}

\begin{cor}
\label{mult-one}
Let $n\ge 3$ be a positive integer, and let $m = \lceil n/2\rceil - 1$.  If $A$ is an abelian variety and $g\in \Aut(A)$ is an automorphism of order $n$
whose type is contained in $\{1/n,2/n,\ldots,m/n\}$, then $A/\langle g\rangle$ has a rational curve.
\end{cor}

\begin{proof}
Let $X$ denote the non-singular projective hyperelliptic curve of genus $m$ which contains the affine curve $y^2 = x^n-1$.
The order-$n$ automorphism $h(x,y) = (e(1/n)x,y)$ extends to an automorphism of $X$ and therefore defines an automorphism of $B := \Jac(X)$.
The Lie algebra $\Lie(B)$ can be identified with the space $H^0(X,\Omega_X)$ of holomorphic differential forms on $X$, which has a basis
$$\Bigl\{\frac {dx}y, \frac{x\,dx}y,\ldots,\frac{x^{m-1}\,dx}y \Bigr\}.$$
Therefore, the type of $h$ acting on $B$ is $\{1/n,2/n,\ldots,m/n\}$, which is disjoint from the type of $h^{-1}$.
On the other hand, $B/\langle h\rangle$ contains the rational curve $X/\langle h\rangle$.
Thus, Theorem~\ref{disjoint} applies.
\end{proof}

\begin{cor}
\label{seven}
If $A$ is an abelian $3$-fold and $g$ is an automorphism of $A$ of type $(1/7,2/7,4/7)$, then $A/\langle g\rangle$ has a rational curve.
\end{cor}

\begin{proof}

Let $X$ denote the Klein quartic:
$$X : x^3y+y^3z+z^3x=0$$
and $B$ the Jacobian of $X$.
The self-map $h(x,y,z)=(\zeta_7x, \zeta_7^4y, \zeta_7^2z)$ of $X$ belongs to the automorphism group
$\PSL_2(\bbf_7)$ of $X$ which acts non-trivially on the Jacobian variety $B$ and therefore on $\Lie(B) = H^0(X,\Omega_X)$.
Conjugating $h$ by the cyclic permutations of $(x,y,z)$, we see that $h$ is conjugate to $h^2$ and $h^4$
in $\Aut(X)$, and therefore the type of $h$ is invariant under multiplication by $2$ (mod $1$).
It is therefore $(1/7,2/7,4/7)$ or $(3/7,5/7,6/7)$, and
replacing $h$ by $h^{-1}$ if necessary, we may assume that it is the former.
\end{proof}

We remark that $B/\langle h\rangle$ in the proof of Corollary~\ref{seven} has appeared
in the literature; it is known to have  a Calabi-Yau resolution  \cite[Example~6.3]{BT2}.

\begin{cor}\label{fifteen}
If $A$ is an abelian variety and $g$ an automorphism such that the type of $A$ is contained in $(1/15,2/15,3/15,4/15,8/15,9/15)$,
then $A/\langle g\rangle$ has a rational curve.
\end{cor}

\begin{proof}
Let $X$ be the non-singular projective curve which has a (singular, affine) model $X'\colon y^{15} = x^2(x-1)$.  This is singular only at $(0,0)$, and the inverse image of this singularity
under the normalization map $X\setminus \{P_\infty\}\to X'$ is a single point $P_0\in X$.
The automorphism $h\colon (x,y)\mapsto (x,e(1/15)y)$ of the affine curve induces an automorphism of $X$ of order $15$.
As $15y^{14}dy = (3x^2-2x)dx$, any differential form $\frac {x^m y^n  dy}{3x^2-2x}$, $m,n\ge 0$, is holomorphic except possibly at $P_0$ and $P_\infty$.
One checks that
$$\frac{dy}{3x-2}, \frac{ydy}{3x-2}, \frac{y^2dy}{3x-2}, \frac{y^3dy}{3x-2},  \frac{y^7dy}{3x^2-2x},   \frac{y^8dy}{3x^2-2x}$$
are all holomorphic, and their eigenvalues under $h$ are $e(1/15)$, $e(2/15)$, $e(3/15)$, $e(4/15)$, $e(8/15)$, $e(9/15)$ respectively.
Applying the Riemann-Hurwitz theorem to the map $X\to \bbp^1$ given by $y$, we see that $X$ is of genus $6$, and therefore, that these differential
forms form a basis of $\Lie(\Jac(B)) = H^0(X,\Omega_X)$.
\end{proof}

\begin{thm}
\label{small-factor}
If $A$ is an abelian variety and $g$ is an automorphism of finite order such that $g$ and $g^{-1}$ have disjoint types and the type of $g$ is a sum
of primitive types at least one of which has $weight$ less than $1$, then $A/\langle g\rangle$ has a rational curve.
\end{thm}

\begin{proof}
For every primitive type, there exists an abelian variety $B_i$ with complex multiplication and an automorphism $h_i$ of $B_i$ with the given type.
Indeed, the primitive types of denominator $n$ are in natural correspondence with CM-types on $\bbq(\zeta_n) = \bbq(e(1/n))$.  Any CM-type $\Phi$ on $\bbq(\zeta_n)$, $n\ge 3$
defines an embedding $\bbq(\zeta_n) \to \bbc^{\phi(n)/2}$.  The image of $\bbz[\zeta_n]$ defines a lattice $\Lambda\subset \bbc^{\phi(n)/2}$, and the quotient $\bbc^{\phi(n)/2}/\Lambda$
is a complex torus with a natural action of $\bbz/n\bbz$ of the type associated with $\Phi$.  The quotient $\bbc^{\phi(n)/2}/\Lambda$ admits a polarization \cite[II~6~Theorem 4]{S}, so there
exists a pair $(B_i,h_i)$ as claimed.  If $\age(h_1) < 1$, then by \cite{KL}, $B_1/\langle g_1\rangle$ has a rational curve.
If $A = A_1\times\cdots\times A_m$, and $g = (g_1,\ldots,g_m)$ is a finite order automorphism of $A$ which stabilizes each factor, then $A_1/\langle g_1\rangle \subset A/\langle g\rangle$,
so $A/\langle g\rangle$ has a rational curve.  The theorem now follows from Theorem~\ref{disjoint}.
\end{proof}

\

To summarize, we have the following theorem:

\begin{thm}\label{main2}
Let $A$ be an abelian variety with a nontrivial automorphism $g$ of finite order such that $\age(g)\le 1$.  Then $A/\langle g\rangle$ contains a rational curve.
\end{thm}

\begin{cor}\label{hyper} Let $A$ be an abelian variety of dimension $n$ with a nontrivial automorphism $g$ of finite order. If $\dim(\ker(1-g))\geq n-2$ (i.e. the codimension of the fixed subspace of $A$ under $g$ is less than or equal to $2$), then the quotient $A/\langle g\rangle$ contains a uniruled hypersurface.
\end{cor}

\begin{proof} Since $\dim(\ker(1-g))\geq n-2$, $B:=\mathrm{im}(1-g)$ is an abelian variety of dimension $n-\dim (\ker (1-g))\le 2$.
Let $h$ denote the restriction of $g$ to $B$.  As $\age(h)+\age(h^{-1}) \ge 2$, we may assume without loss of generality that $\age(h)\le 1$,
so $B/\langle h\rangle$ has a rational curve $Z$ by Theorem~\ref{main2}.
  Let $C$ denote the identity component of $\ker (1-g)$, which is an abelian subvariety of dimension $\dim \ker(1-g)$ on which $g$ acts trivially.  The addition morphism $B\times C\to A$ is an isogeny and respects the action of  $\langle g\rangle$.  We therefore obtain a finite morphism from
$$Z\times C \subset (B/\langle g\rangle)\times C \cong (B\times C)/\langle g\rangle$$
to $A/\langle g\rangle$.  The image of an $n-1$-dimensional ruled variety under a finite morphism is a uniruled hypersurface.
\end{proof}

\begin{cor} Let $E$ be an elliptic curve. If $W$ is a Weyl group of simple roots of rank $n\geq 3$ acting on $E^n$ and $W^+$ is an index $2$-subgroup of $W$, then the quotient $E^n/W^+$ contains a rational curve.
\end{cor}
\begin{proof} $W$ is generated by reflections $s_j$ of simple roots. Since $W^+$ is an index $2$-subgroup of $W$, there exist  two reflections $s_1$ and $s_2$ such that $s_1s_2\in W^+$. Then for each $i$, $\ker(1-s_i)$ has codimension $1$ and their intersection has codimension $\leq 2$. Since $\ker(1-s_1s_2)$ contains the intersection of $\ker(1-s_1)$ and $\ker(1-s_2)$,  this follows from Corollary~\ref{hyper}.
\end{proof}


\begin{thebibliography}{99}

\bibitem{BT} Bogomolov, Fedor;  Tschinkel, Yuri: Density of rational points on elliptic K3 surfaces, \textit{Asian~J.~Math.} \textbf{4},  (2000), 351--368,

\bibitem{BT2} Bogomolov, Fedor; Tschinkel, Yuri: Rational curves and points on K3 surfaces. \textit{Amer.\ J.\ Math.} \textbf{127} (2005), no.\ 4, 825--835.


\bibitem{KL} Koll\'{a}r, J\'anos; Larsen, Michael: Quotients of Calabi-Yau varieties. Algebra, arithmetic, and geometry: in honor of Yu. I. Manin. Vol. II, 179--211, Progr. Math., 270, Birkhauser Boston, Inc., Boston, MA, 2009.

\bibitem{La} Lang, Serge: Abelian varieties, New York: Springer-Verlag, 1983.

\bibitem{Larsen} Larsen, Michael:
Rank of elliptic curves over almost algebraically closed fields. \textit{Bull.\ London Math.\ Soc.} \textbf{35} (2003), no.\ 6, 817--820.

\bibitem{L} Looijenga, Eduard: Root systems and elliptic curves. \textit{Invent.\ Math.} \textbf{38} (1976/77), no. 1, 17--32.

\bibitem{P} Pirola, Gian Pietro: Curves on generic Kummer varieties. \textit{Duke Math.\ J}.\textbf{ 59} (1989), no. 3, 701--708.

\bibitem{R} Reid, Miles: Chapters on algebraic surfaces, Complex algebraic geometry (Park City, UT, 1993), 3--159, IAS/Park City Math. Ser., 3, Amer. Math. Soc., Providence, RI, 1997.

\bibitem{S} Shimura, Goro: Abelian Varieties with Complex Multiplication and Modular Functions. Princeton Mathematical Series, 46. Princeton University Press, Princeton, NJ, 1998.
 \end{thebibliography}
\end{document}